\documentclass[11pt]{amsart}
\usepackage[english]{babel}
\usepackage[utf8]{inputenc}
\usepackage{mathrsfs}
\usepackage{mathtools,soul,bbm}
\usepackage{graphicx,subfigure}
\usepackage{amsmath,amsthm,amssymb,amsfonts,dsfont}
\usepackage{bigints}

\usepackage{hyperref}
\hypersetup{
    colorlinks=false,
    urlcolor=black,
    linkcolor=black,
    breaklinks=true
}

\usepackage{geometry}
\usepackage{xcolor}
\usepackage[toc,page]{appendix}
\newtheorem{thm}{Theorem}[section]
\newtheorem*{thm*}{Theorem}
\newtheorem{lem}[thm]{Lemma}
\newtheorem{prop}[thm]{Proposition}
\newtheorem{defn}[thm]{Definition}
\newtheorem{cor}[thm]{Corollary}

\theoremstyle{definition}
\newtheorem{rem}[thm]{Remark}
\newtheorem{ex}[thm]{Example}
\newtheorem{conj}[thm]{Conjecture}

\newtheorem{alphatheorem}{Theorem}

\def\R{{\mathbb{R}}}

\begin{document}

\title{Functional Liftings of Restricted Geometric Inequalities}

\author[A. Malliaris]{Andreas Malliaris}
\address{Institut de Math\'ematiques de Toulouse (UMR 5219). University of Toulouse \& CNRS. UPS, F-31062
Toulouse Cedex 09, France.}
\email{andreas.malliaris@math.univ-toulouse.fr}

\author[J. Melbourne] {James Melbourne}
\address{Department of Probability and statistics, Centro de Investigacion En Matem\'aticas ´
(CIMAT), Mexico.}
\email{james.melbourne@cimat.mx}

\author[C. Roberto]{Cyril Roberto}
\address{MODAL’X, UPL, Univ. Paris Nanterre, CNRS, F92000 Nanterre France.}
\email{croberto@math.cnrs.fr}

\author[M. Roysdon]{Michael Roysdon}
\address{Department of Mathematics, Applied Mathematics, and Statistics, Case Western Reserve University, Cleveland, OH 44106, USA \& Department of Mathematical Sciences, University of Cincinnati, Cincinnati, OH 45221, USA.}
\email{roysdoml@ucmail.uc.edu}

\begin{abstract}
    We investigate what we term ``generalized sup-convolutions".  We show that  functional inequalities that enjoy an interpretation as sup-convolution inequalities can be deduced from the special case of indicator functions corresponding to a geometric inequality.  As consequences we derive a Borell-Brascamp Lieb inequality for the Gaussian Brunn-Minkowski inequality and give a functional analog of the log-Brunn-Minkowski conjecture. Though we focus on Euclidean applications, our results are general and can be directly applied in more abstract settings, like groups or even topological measure spaces without algebraic structure, we instantiate this claim with a Borell-Brascamp-Lieb type inequality for nilpotent Lie groups.
\end{abstract}
\maketitle

\section{Introduction}

    In this article we investigate the equivalence of certain types of geometric inequalities and their functional analogs, each of which can be reviewed as reverse H\"older type inequalities.  Our setting will be abstract and rather flexible, but let us start with an example: the well-known equivalence of the celebrated Brunn-Minkowski inequality (see \cite{Sh1}), a linchpin for modern convex geometry, and its functional version, the  well-known Pr\'ekopa-Leindler inequality (see \cite{Prekopa,Leindler}),  possessing broad applications in various fields throughout mathematics. We now present these two inequalities.
    
    The Brunn-Minkowski inequality asserts that, for any pair of Borel sets $A, B \subset \R^n$, the following volume inequality 
    \begin{align} \label{eq: BMI}
        |A + B|^{\frac 1 n} \geq |A|^{\frac 1 n} + |B|^{\frac 1 n},
    \end{align} 
    where  $|\cdot|$ (or $|\cdot|_n$) is the $n$-dimensional Lebesgue volume, is valid. 
    The Pr\'ekopa-Leindler inequality asserts that for any triple of non-negative measurable functions $f, g$ and $h$ on $\mathbb{R}^n$ such that
    \begin{align} \label{eq: PLI hypothesis}
        h((1-t)x + ty) \geq f^{1-t}(x) g^t(y), \qquad \mbox{for some } t \in [0,1] \mbox{ and all } x,y \in \mathbb{R}^n 
    \end{align}
    the integrated inequality
 \begin{equation}
     \label{eq:PL}
          \int h \geq \left( \int f \right)^{1-t} \left( \int g \right)^t
 \end{equation}
 holds.

    The numerous extensions and the ramifications of the Brunn-Minkowski and Pr\'ekopa-Leindler inequality has had in various fields of mathematics are detailed in the survey of Gardner \cite{G20}.
    An exhaustive list of references is beyond the scope of this paper. We would like to highlight two very important and classical papers of Borell \cite{borell1975convex}, and Brascamp and Lieb \cite{brascamp1976best}, where generalizations of both of these inequalities have appeared; such inequalities have become referred to as ``Borell-Brascamp-Lieb" and ``Brascamp-Lieb" inequalities. More recent examples of such works, as well as modern treatments, can be found in the books \cite{AGA,AGA2,Sh1}.

   We now return to our original point of investigation: the equivalence of the inequalities \eqref{eq: BMI} and \eqref{eq:PL}. 
   Taking $f = \mathbbm{1}_A$ and $g = \mathbbm{1}_B$, in \eqref{eq:PL}, yields 
\begin{equation}
    \label{eq:BMdimensionfree}
    |(1-t) A + tB| \geq |A|^{1-t}|B|^t
\end{equation}
    which, after a clever choice of $t$, recovers \eqref{eq: BMI}.  The converse direction can be proven by showing that Brunn-Minkowski implies Pr\'ekopa-Leindler in one-dimension and then using the tensorization property of \eqref{eq:PL} to recover the $n$-dimensional functional inequality, or alternatively interpreting the integrals in \eqref{eq:PL} as the volumes of $(n+1)$-dimensional sets.  

    However, in more general settings when a geometric inequality holds only for a restricted class of sets and/or a particular measure on a vector space or group, it can be less clear how to implement such approaches.  For example, the conjecture of Gardner and Zvavitch \cite{GZ10} asked if \begin{equation}
\label{eq:GaussianDimensional}
\gamma((1-t)A+tB) \geq\left [(1-t) \gamma(A)^{\frac{1}{n}} + t \gamma(B)^{\frac{1}{n}}\right]^n
\end{equation}
 held for $A,B$ origin-symmetric convex bodies in $\R^n$ (see Section~\ref{sec:Abstract}) and $\gamma$ the standard Gaussian measure on $\mathbb{R}^n$. By Borell \cite{borell1975convex}, an $n$-dimensional probability measure that satisfies \eqref{eq:GaussianDimensional} for all $A,B$ Borel sets is necessarily the uniform measure on a bounded convex set, thus such an inequality necessarily fails for $\gamma$ and general $A,B$. Nevertheless this question was answered in the affirmative by Eskenazis and Moschidis \cite{eskenazis2021dimensional}.

  More generally, it is conjectured that an improvement like \eqref{eq:GaussianDimensional}, in the class of symmetric convex bodies, holds for any even log-concave measure $\mu$ on $\mathbb{R}^n$. The most general result in this direction is due to Cordero-Erausquin and Rotem \cite{cordero2023improved}, where it is proved that \eqref{eq:GaussianDimensional} holds for rotationally invariant measures $d\mu(x)=e^{-w(|x|)}dx,$ where $w:[0,\infty)\to (-\infty,\infty]$ is increasing with $\mathbb{R}\ni t \mapsto w(e^t)$ convex. Note that this class goes beyond log-concavity as it contains, e.g., the Cauchy measure.

 Our first result is a functional version of the result in \cite{cordero2023improved}. We say that a function $f \colon \R^n \to \R$ unimodal if its super-level sets $\{f \geq t\}$ are convex bodies (see Section~\ref{sec:general}).

 \begin{thm} \label{thm: rotationally invariant BBL for unimodal}
     Let $\mu$ be the measure on $\mathbb{R}^n$, with density $e^{-w(|x|)}$, where $w:[0,\infty)\to (-\infty,\infty]$ is increasing and $\mathbb{R}\ni t \mapsto w(e^t)$ is convex. Let $t\in (0,1)$, $\alpha\geq -\frac{1}{n}$. If $f,g:\mathbb{R}^n\to [0,\infty)$ are even unimodal and $h:\mathbb{R}^n\to [0,\infty)$ is Borel measurable satisfying 
     \begin{equation}
         h((1-t)x+ty) \geq [(1-t)f(x)^\alpha + tg(y)^\alpha]^\frac{1}{\alpha}, \qquad \qquad \forall x,y\in \mathbb{R}^n,
     \end{equation}
     then 
     \begin{equation} \label{eq:Theorem1.1}
         \int_{\mathbb{R}^n}h\ d\mu \geq \left( (1-t)\left(  \int_{\mathbb{R}^n}f\ d\mu\right)^\beta+t\left( \int_{\mathbb{R}^n}g\ d\mu \right)^\beta  \right)^{\frac{1}{\beta}},
     \end{equation}
     where $\beta=\frac{\alpha}{1+n\alpha}$. Moreover, \eqref{eq:Theorem1.1} is equivalent to the geometric inequality
\begin{equation} \label{dimensional BM rot inv}
    \mu\left((1-t)K+tL\right)^{\frac{1}{n}}\geq (1-t)\mu(K)^{\frac{1}{n}}+t\mu(L)^{\frac{1}{n}},
\end{equation}
     for all $K,L\subset\R^n$ symmetric convex bodies.
 \end{thm}

Very recently, Cordero-Erausquin and Eskenazis \cite{CE2025} found a functional analogue of \eqref{dimensional BM rot inv} by following a novel geometric way to use the so-called $L^2$-method. In Section \ref{sec:mesures}, we explain how Theorem \ref{thm: rotationally invariant BBL for unimodal} implies their main result, while answering \cite[Question 19]{CE2025} in the radially symmetric case.

In another recent work \cite{AD25} functional analogues of the said geometric inequalities for smaller class of measures were proven via completely different tools coming from information theory. Specializing Theorem~\ref{thm: rotationally invariant BBL for unimodal}  to the Gaussian measure we are able to generalize \cite[Theorem~1.7]{AD25} when $f,g$ are assumed to be $\log$-concave and $\alpha \geq 0.$
 

\begin{cor}\label{thm: Gaussian Borell Brascamp Lieb} Let $\gamma$ be the standard Gaussian probability measure on $\R^n$,  $t \in [0,1]$ and $\alpha \geq - \frac{1}{n}$. Let  $f,g \colon \R^n \to [0,\infty)$ be even, unimodal functions. Assume that $h \colon \R^n \to [0,\infty)$ is any Borel measurable function for which 
\[
h((1-t)x+ty) \geq [(1-t)f(x)^\alpha + tg(y)^\alpha]^\frac{1}{\alpha}
\]
holds for all $x,y \in \R^n$. Then 
\[
\int_{\R^n}h d\gamma \geq \left[(1-t) \left(\int_{\R^n} f d \gamma\right)^\beta + t   \left(\int_{\R^n} g d \gamma\right)^\beta\right]^\frac{1}{\beta},
\]
where $\beta= \frac{\alpha}{1+n\alpha}$. Moreover, this functional inequality is equivalent to \eqref{eq:GaussianDimensional}. 
\end{cor}

We also recover \cite[Theorem~1.8]{AD25}, a functional analogue of a dimensional Brunn--Minkowski type inequality for star bodies, see Section \ref{sec:mesures}.


Our approach\footnote{We direct the reader to \cite{DarioNote} for other approaches.} can be traced back at least to the 1956 proof of the Brunn-Minkowski inequality by Hadwidger and Ohmann.  It was their ``slicing and balancing'' of sets that we adapted to functions.  In particular to what we will refer to as sup-convolutions\footnote{Note that this a multiplicative version of what is usually called a sup-convolution, when $f = e^{-V}, g = e^{-W}$, 
    \[
        \sup_{(1-t) x + ty = z} f^{1-t}(x) g^t(y) = \exp \left\{ - \inf_{(1-t)x + ty = z} \bigg[ (1-t) V(x) + t W(y) \bigg] \right\}.
    \]
    The quantity  $\inf_{(1-t)x + ty = z} \bigg[ (1-t) V(x) + t W(y) \bigg]$ is  referred to as the inf-convolution of $V$ and $W$. }. For instance in the $\alpha = 0$ case, given a pair of non-negative measurable functions $f,g \colon \R^n \to [0,\infty)$, we set
\begin{equation}
\label{eq:operationexample}
f \square g(z) \coloneqq \sup_{\genfrac{}{}{0pt}{}{x,y:}{(1-t) x + ty = z}} f^{1-t}(x) g^t(y).
\end{equation}
    That is, $f \square g$ is the smallest function that satisfies \eqref{eq: PLI hypothesis}. Indeed, any $h$ satisfying \eqref{eq: PLI hypothesis}, satisfies $h(z) \geq f \square g(z) \geq f^{1-t}(x) g^{t}(y)$ for $(1-t)x +ty = z$, so that (neglecting potential measurability issues) we have 
    \[
        \int h \geq \int f\square g \geq \left( \int f \right)^{1-t} \left( \int g \right)^{t}.
    \]
    A property of the $\square$ operation that proves pivotal, beyond some reasonably expected monotonicity and measurability (see Section \ref{sec:general} for a complete development), is a somewhat peculiar, restricted reverse triangle inequality. That is given two pairs of functions $(f_0,g_0)$ and $(f_1,g_1)$ such that $f_1(x) > 0$ implies that $f_0(x) \geq f_0(y)$ for all $y \in \mathbb{R}^n$, and $g_1(x) >0$ implies $g_0(x) \geq g_0(y)$ for all $y \in \mathbb{R}^n$, then
     \[
         (f_0 + f_1) \square (g_0 + g_1) \geq f_0 \square g_0 + f_1 \square g_1.
     \]

A famous conjecture in the geometric of convex bodies is the $\log$-Brunn-Minkowski conjecture put forth by B\"or\"oczky, Lutwak, Yang and Zhang in \cite{BLYZ12}: given a pair of origin symmetric convex body $K,L \subset \R^n$ and $t \in (0,1)$, is it necessarily true that 
\begin{equation}\label{e:logBM}
|K^{1-t}L^t| \geq |K|^{1-t}|L|^t?
\end{equation}
Here $K^{1-t}L^t$ is the $L_0$-Minkowski combination given by 
\[
K^{1-t} L^t :=\bigcap_{\theta \in S^{n-1}} \left\{x \in \R^n \colon \langle x,\theta\rangle \leq h_K(\theta)^{1-t}h_L^t(\theta)\right\}
\]
where $h_K, h_L$ are the support functions of $K$ and $L$ respectively, defined in \eqref{eq:support}.
Owing to the inclusion $K^{1-t}L^t \subset (1-t)K+tL$, and the homogeneity of the volume, it is readily clear that \eqref{e:logBM} implies \eqref{eq: BMI}. For more details see Section~\ref{sec:lp}. 

Our main contribution to the $\log$-Brunn-Minkowski conjecture is the following equivalence between the inequality \eqref{e:logBM} and what we call the ``$\log$-Pr\'ekopa-Leindler'' inequality.

\begin{thm}\label{thm:logPLintro} Let $t \in [0,1]$. Then the Lebesgue measure satisfies inequality \eqref{e:logBM} for all origin symmetric convex bodies $K$ and $L$, if and only if, for a triple of integrable functions $f,g,h  \colon \R^n \to [0,\infty)$, with $f,g$ even and unimodal functions, such that
\[
h(z) \geq  f(u)^{1-t} g(v)^t,
\]
for $u,v \in C(z;f,g)$, where
\[
C(z;f,g) \coloneqq \{(u,v) \colon z \in \{f \geq f(u)\}^{1-t}  \{g \geq g(v)\}^t\},
\]
implies
\begin{equation} \label{e:logPL}
    \int_{\mathbb{R}^n} h \geq \left( \int_{\mathbb{R}^n} f \right)^{1-t} \left( \int_{\mathbb{R}^n} g \right)^t.
\end{equation}
\end{thm}
That is, Theorem \ref{thm:logPLintro} gives an equivalent functional version of the celebrated $\log$-Brunn-Minkowski conjecture of B\"or\"oczky, Lutwak, Yang, and Zhang \cite{BLYZ12}. In particular, Theorem~\ref{thm:logPLintro} affirmatively answers a question posed by Andrea Colesanti ``Is there a functional version of the log-Brunn-Minkowski inequality?'' at the conference {\it 61 Probability Encounters, in honour of Sergey Bobkov} held at Paul Sabatier University, Toulouse, France in the summer of 2023 \cite[Main Question]{ColesantiSlides}.   

Inequality \eqref{e:logBM} has been verified in several cases. In particular, it has been verified for $n=2$ in the original paper \cite{BLYZ12}. As consequence of their result and that of Theorem~\ref{thm:logPLintro}, the inequality \eqref{e:logPL} holds when $n =2$ for all even unimodal functions.


In general, our techniques extend inequalities for classes of sets to the class of functions whose sup level sets belong to said sets.  For instance in the unconditional case\footnote{A function $f \colon \R^n \to \R$ is said to be unconditional if, for each $\varepsilon_i \in \{-1,1\}$ $i=1,\dots,n$, and $x=(x_1,\dots,x_n) \in \R^n$, $f(\varepsilon_1x_1,\dots,\varepsilon_n x_n) = f(x)$; a set is unconditional if its characteristic function is an unconditional function.}, where the log-Brunn-Minkowski conjecture has been verified (see \cite{CS15}), we have the  equivalence\footnote{We thank Franck Barthe for the observation that inequality \eqref{e:logPL} can be deduced from \eqref{eq:PL} by a well known change of variables (see \cite{CS15} for more details) in the unconditional case. 
} between the $\log$-Pr\'ekopa-Leindler inequality for unconditional functions, inequality \eqref{e:logPL}, and the $\log$-Brunn-Minkowski \eqref{e:logBM} for $K$ and $L$ unconditional.

As mentioned in the abstract the techniques developed are not limited to Euclidean spaces.  In fact, the ideas here are inspired by, and could be used to recover, the proof of a Pr\'ekopa-Leindler type inequality on the integers, by the second author and Marsiglietti \cite[Theorem 3.8]{marsiglietti2024geometric}
(see also \cite{klartag-lehec,GRST} for more on discrete Prekopa-Leindler-type inequalities).  To further hint at broader application, we also give a Borell-Brascamp-Lieb inequality for nilpotent groups, generalizing Tao \cite{Tao2011}, which in turn generalized the case given for the Heisenberg group \cite{Monti03}.
\begin{thm}
    Let $G$ be a simply connected nilpotent Lie group of (topological) dimension $d$ with a Haar measure $\mu$ and $\mathcal{M}_p^{(t)}(u,v)$ the $t$-weighted $p$-mean\footnote{We refer to the next section and more precisely to \eqref{eq:pmeans} for the definition of $\mathcal{M}_{p}^{(t)}$.} of $u$ and $v$. Then, for $\alpha \in  [- \frac 1 d,1]$ and $t \in (0,1)$, if
    \[
        h(xy) \geq \mathcal{M}_\alpha^{(t)}( f(x), g(y))
    \]
    holds for all $x,y$, then
    \[
\int_G h \ d\mu \geq \mathcal{M}_{\alpha'}^{(t)} \left( \frac{\int_G  f \ d\mu}{(1-t)^d}, \frac{\int_G g \ d\mu}{t^d} \right).
    \]
    where $\alpha' = \frac{\alpha}{1 + \alpha d}$.
\end{thm}

Let us outline the remainder of the article.  In Section \ref{sec:general} we introduce an abstract sup-convolution, and derive our main technical tool in Theorem \ref{thm: everything is sup convolution}.  In Section~\ref{sec:mesures} we derive the dimensional Brunn-Minkowski inequality for the Gaussian measure and generalizations thereof.  
In Section~\ref{sec:lp} we give a functional version of the log-Brunn-Minkowski conjecture and related results for the $L_p$-Brunn-Minkowski theory.   Finally in Section \ref{sec:Abstract} we discuss further generalizations of the development in Section~\ref{sec:general}, limitations of our methods, and give a Borell-Brascamp-Lieb type inequality for nilpotent Lie groups.

\section{Generalized Sup-Convolutions} \label{sec:general}

In this section we introduce some well-known objects from convex geometry (convex body, means) and the new notion of generalized or abstract sup-convolution (Definition \ref{def: sup convolution}) and prove two related abstract theorems that constitute the starting point of our investigations.

\medskip

A set $K \subseteq \mathbb{R}^n$ is a body, if $K$ is compact with non-empty interior.  $K$ is a convex body, if it is body that satisfies $(1-t)K + tK \subseteq K$ for $t \in [0,1]$. We call a function $f: \mathbb{R}^n \to \mathbb{R}$ unimodal if $$\{f \geq t \} = \{x \in \mathbb{R}^n: f(x) \geq t \}$$ is a convex body for $t \in \mathbb{R}$ when it is non-trivial (that is non-empty and not all of $\mathbb{R}^n$).  Many of the results to come will follow with lighter assumptions, for instance, compactness can often be replaced by Borel measurability, but we will not pursue these technicalities. 

For $u = (u_1, \dots, u_n)$ with $u_i >0$, $t = (t_1, \dots, t_n)$ with $t_i >0$ and $p \in (-\infty, 0) \cup (0,\infty)$ we define the ``mean" ($\sum_i t_i$ is not necessarily set to be one)
\begin{equation}
\label{eq:pmeans}
\mathcal{M}_p^{(t)}(u) = \left( \sum_{i=1}^n t_i u_i^p \right)^{\frac 1 p}.
\end{equation}
For $p \in \{-\infty, 0 , \infty \}$ we define
\[
    \mathcal{M}_{-\infty}^{(t)}(u) \coloneqq \min_i \{ u_i \}, 
    \qquad \qquad 
    \mathcal{M}_0^{(t)}(u) \coloneqq \prod_{i=1}^n u_i^{t_i},  
    \qquad \qquad 
    \mathcal{M}_\infty^{(t)} (u) \coloneq \max_i \{u_i\}.
\]
Note that for $p \in \{\pm \infty\}$ this is the continuous extension of $\mathcal{M}_p^{(t)}$; while under the assumption that $\sum_{i=1}^n t_i = 1$ it is the continuous extension for $p =0$ as well.
We further extend $\mathcal{M}_p^{(t)}$ to $[0,\infty)^n$ by 
\[
    \mathcal{M}_{p}^{(t)}(u) = 0
\]
for $u$ such that $\prod_{i=1}^n u_i = 0$.

For later purpose, let us observe that H\"older's inequality gives 

\begin{equation} \label{eq:holder-mean}
 (p+q) \mathcal{M}_{p}^{(t)}( u) \mathcal{M}_{q}^{(t)}(v) \geq (p+q) \mathcal{M}_{r}^{(t)}(uv),   
\end{equation}
where $(uv)_i = u_iv_i$,  for $p,q,r \in \mathbb{R}$, satisfying $\frac{1}{r}=\frac{1}{p}+\frac{1}{q}$.  Note, all of the inequalities are trivial if either $u$ or $v$ have $0$ in any coordinate, and that $\frac{1}{r} = \frac{1}{p} + \frac{1}{q}$ has no solution in $r$ when $ p+q = 0$. However, when $p+q>0$ we have $$\mathcal{M}_{p}^{(t)}(u) \mathcal{M}_{q}^{(t)}(v) \geq \mathcal{M}_{r}^{(t)}(uv)$$ and taking $q >0$ and $p \downarrow -q$, so that $r \downarrow - \infty$ we have by continuity,
\[
    \mathcal{M}_{-q}^{(t)}(u) \mathcal{M}_q^{(t)}(v) \geq \mathcal{M}_{-\infty }^{(t)}(uv).
\]
This last fact is also a consequence of the inequality $\sum_i t_i x_i \leq C \sum_i t_i y_i $, for positive numbers $x_i \leq C y_i$.

For notational convenience we may sometimes use $\mathcal{M}_{p}^{(t)}(u_i)$ instead of $\mathcal{M}_{p}^{(t)}(u)$.

\medskip

Given a collection of sets, the next definition constructs a corresponding collection of functions.

 \begin{defn}
     Given a Polish measure space $(E, \mathscr{E},\mu)$ and a collection of subsets $\mathcal{A} \subseteq \mathscr{E}$ containing both the empty set, we let $\mathcal{F}(\mathcal{A})$ denote the collection of functions
     \[
        \mathcal{F}(\mathcal{A}) \coloneqq \{ f: E \to \mathbb{R}  \ | \ \{f \geq t \} \in \mathcal{A} \hbox{ for all } t \in \mathbb{R} \}.
     \]
 \end{defn}

 Hereafter, we will assume that a collection of sets $\mathcal{A}$ contains the empty set, even if it is not explicitly mentioned.  For example, we may say, let $\mathcal{A}$ be the collection of origin symmetric convex bodies, in place of let $\mathcal{A}$ be the collection of origin symmetric convex bodies and the empty set. 
 
 For $p >0$ and $f \in \mathcal{F}(\mathcal{A})$ we denote $\|f\|_p \coloneqq \left( \int |f|^p d\mu \right)^{\frac 1 p}$ and emphasize that the underlying measure $\mu$ is understood.


Let $\{(E_1,\mathscr{E}_1, \mu_1), \dots, (E_n, \mathscr{E}_n, \mu_n)\}$ and $(F,\mathscr{F},\nu)$ be measure spaces.  We will denote $E = E_1\times \cdots \times E_n$ and $\mu = \mu_1 \otimes \cdots\otimes\mu_n$.  For $f = (f_1, \dots,f_n)$ with $f_i :E_i \to \mathbb{R}$, and $x \in E$, set
\[
    f(x) \coloneqq (f_1(x_1), \dots, f_n(x_n)) .
\]
For vectors $u, v \in \mathbb{R}^n$ we write $u \leq v$ when $u_i \leq v_i$ for all $i$.  We assume that there exists on each $E_i$ a privileged class of sets $
 \mathcal{A}_i$, and we will consider a generalized sup-convolution of $f_i \in \mathcal{F}(\mathcal{A}_i)$. 

\medskip

We are now in position to introduce the generalized sup-convolution that constitutes one of the central objects of this paper.
 
\begin{defn} \label{def: sup convolution}
    We consider a functional $\square$ that takes a collection  $f = (f_1, \dots, f_n)$ of non-negative $f_i \in \mathcal{F}(\mathcal{A}_i)$ to $\square f : F \to [0,\infty]$ to be a \emph{generalized sup-convolution} if the following are satisfied for $f_i,g_i \in \mathcal{F}(\mathcal{A}_i)$:
    \begin{enumerate}
        \item \label{item: monotonicity of sup conv} $f_i \leq g_i$ for all $i$ implies $\square f \leq \square g$ 
        \item \label{item: super additivity on sliced functions} 
        If each $g_i$ has the property that $g_i(x) > 0$ implies then $f_i(x) \geq f_i(y)$ for all $y \in E_i$, it holds that 
        $\square{(f + g)} \geq \square f  + \square g$;
        \item \label{item: sup-conv preserves step functions} if $s = (a_1\mathbbm{1}_{A_1}, \dots, a_n\mathbbm{1}_{A_n})$, for $A_i \in  \mathcal{A}_i$, then its sup-convolution $\square s$ is $\mathscr{F}$-measurable.
    \end{enumerate}
\end{defn}

Note that our notion of a sup-convolution is dependent on the measure spaces $(E_i,\mathscr{E}_i,\mu_i)$ and $(F, \mathscr{F},\nu)$, but it does not necessarily depend on an algebraic structure of said spaces.  

For brevity we will hereafter refer to a ``generalized sup-convolution'' as just a {\it sup-convolution}.  Recall that if $E_1 = E_2 = \mathbb{R}^n$ and $\mathcal{A}_1 = \mathcal{A}_2 = \mathcal{B}(\mathbb{R}^n)$ the Borel $\sigma$-algebra, and $f = (f_1, f_2) \in \mathcal{F}(\mathcal{A}_1 \times \mathcal{A}_2)$
\[
    \square f(z) \coloneqq \sup_{z = (1-t)x + ty } f_1^{1-t}(x) f_2^t(y)
\]
is a sup-convolution. 

More generally, if $E_1,E_2, \dots, E_n, F$ are Polish spaces, $t_i > 0$, $\sum_{i=1}^n t_i =1$ and $\Phi: E\to F$ continuous, then
\[
    \square f(z) \coloneqq \sup_{z = \Phi(x)} \prod_{i=1}^n f_i^{t_i }(x_i)
\]
is a sup-convolution as well. This is a very particular case of the next result.



In what follows a set is said to be analytic if it is the continuous image of a Polish space.  

\begin{thm} \label{thm: everything is sup convolution} 

Under the above conventions, given $\mathcal{A} = (\mathcal{A}_1, \dots, \mathcal{A}_n)$ with each $\mathcal{A}_i$ a sub-collection of analytic sets, for $\alpha \leq 1$ and $t= (t_1,\dots,t_n)$, with $t_i >0$ such that $\sum_{i=1}^n t_i =1$
    and $C(z) = \{x : \Phi(x) = z\}$ for a continuous function $\Phi:E =E_1 \times \cdots \times  E_n\to F$. The map defined by
    \[
        \square f(z) \coloneqq \sup_{x \in C(z)} \mathcal{M}_\alpha^{(t)}(f_i(x_i))
    \]
    is a sup-convolution.
\end{thm}

In the following proof (and later on) we call a function a ``step'' function if it takes only finitely many values.

\begin{proof}[Proof of Theorem~\ref{thm: everything is sup convolution}]
    Take $f \leq g$ and $x \in C(z)$. Then $\mathcal{M}_\alpha^{(t)}(f_i(x_i)) \leq \mathcal{M}_\alpha^{(t)}(g_i(x_i)) \leq \square g(z)$.  Taking supremum over $x \in C(z)$ shows that $\square$ is monotone.

    We proceed to verifying superadditivity.  Assume that $f= (f_1,\dots,f_n)$ and $g = (g_1,\dots,g_n)$, with $f_i,g_i$ being functions such that if $g(x_i) > 0$, then $f_i(x_i) \geq f_i(y_i)$ for all $y_i \in E_i$. Now if $\square g(z) = 0$, then the result is trivial from monotonicity.  Otherwise,  $g(z) >0$, and there exists\footnote{Since $g_i$ are step functions $x \mapsto \mathcal{M}_\alpha^{(t)}(g_i(x_i))$ takes only finite many values.} an $x \in C(z)$ such that $\square g(z) = \mathcal{M}_\alpha^{(t)}(g_i(x_i))> 0$.  For such an $x$, necessarily $g_i(x_i) >0$ so that $f_i(x_i) \geq f_i(y_i)$ for all $y_i \in E_i$. It follows that
    \[
        \mathcal{M}_\alpha^{(t)}(f_i(x_i)) = \square f(z).
    \]
    Consequently,
\begin{align*}
        \square f(z) + \square g(z) &= \mathcal{M}_\alpha^{(t)}(f_i(x_i)) + \mathcal{M}_\alpha^{(t)}(g_i(x_i))\\
        &\leq \mathcal{M}_\alpha^{(t)}((f_i + g_i)(x_i))\\
        &\leq \sup_{y \in C(z)} \mathcal{M}_\alpha^{(t)}((f_i + g_i)(x_i))\\
        &= \square(f+g)(z),
\end{align*}
as required.

To establish measurability, given $s = ( a_1 \mathbbm{1}_{A_1}, \dots, a_n \mathbbm{1}_{A_n})$, we have
\[
\square s(z) = \mathcal{M}_\alpha^{(t)}(a_i) \mathbbm{1}_{\Phi(A)}(z).
\]
For each $i$, since $A_i$ is analytic, there exist Polish spaces $P_i$ and continuous functions $\psi_i$ such that $\psi_i(P_i) = A_i$.  Consider the Polish space $P = P_1 \times \cdots \times P_n$, and define $\psi: P \to F$ by
\[
\psi(p) = \Phi(\psi_1(p_1), \dots, \psi_n(p_n)).
\]
Then $\psi(P) = \Phi(A)$, so that $\Phi(A)$ is analytic, and universally measurable (see \cite[Chapter~43]{Fremlinv4}).  It follows that $\square s$ is universally measurable as well. 
 \end{proof}


 Our main tool in this paper is the following equivalence of geometric and functional inequalities.

\begin{thm} \label{thm: abstract sup convolution theorem} Under the above conventions, given $\mathcal{A} = (\mathcal{A}_1, \dots, \mathcal{A}_n)$ with each $\mathcal{A}_i$ a sub-collection of analytic sets, for $\alpha \leq 1$ and $t= (t_1,\dots,t_n)$, with $t_i >0$ such that $\sum_{i=1}^n t_i =1$, a triple $(\nu,\mu, \square)$ and $p = (p_1, \dots, p_n) \in [1,\infty)^n$  satisfies 
    \begin{equation} \label{eq: abstract sup theorem assumption}
        \nu( \square a ) \geq \mathcal{M}_\alpha^{(t)}(\|a_i\|_{p_i})
    \end{equation}
    for all functions $a$ of the form  $a  \coloneqq (a_1 \mathbbm{1}_{A_1}, \dots, a_n \mathbbm{1}_{A_n})$ for $a_i \in (0,\infty)$ and $A_i \in \mathcal{A}_i$ if and only if it satisfies for any collection of non-negative $f_i \in L_{p_i}(\mu_i) \cap \mathcal{F}(\mathcal{A})$,
    \begin{equation}\label{eq: abstract sup theorem conclusion}
        \nu_*(\square f) \geq \mathcal{M}_\alpha^{(t)}(\|f_i\|_{p_i}),
    \end{equation}
    where $\nu_*$ denotes the lower integral of $\square f$.
\end{thm}



Before we prove the theorem, we require the following lemma. 

\begin{lem} \label{lem: f- and f+ satisfy supconv condition}
    For a set $U$ and function $f : U \to \mathbb{R}$ and $t \in \mathbb{R}$, set
    \[
        f^- \coloneqq \min \{ f, t\} \ \ \ \hbox{ and } \ \ \ f^+ \coloneqq \max \{ f -t, 0 \}.
    \]
    If $f^+(x) > 0$, then $f^-(x) \geq f^-(y)$ for all $y \in U$.
\end{lem}
\begin{proof}
    For $x$ such that $f^+(x) > 0$, $f(x) > t$, so that $f^-(x) = t \geq \min\{f(y),t \} = f^-(y)$.  
\end{proof}

\begin{proof}[Proof of Theorem \ref{thm: abstract sup convolution theorem}]
Assume that $p_i=1$, for each $i=1,\ldots,n$. Let us first observe that it is enough to prove the inequality for $f_i$ step functions.  Define
$$f_i^{(m)}     \coloneqq \sum_{k=1}^{2^{2m}} 2^{-m} \mathbbm{1}_{\left\{f_i \geq \frac{k}{2^m} \right\}},$$ 
and set $f^{(m)} = (f_1^{(m)}, \dots, f_n^{(m)})$. Note that $f_i \in \mathcal{F}(\mathcal{A}_i)$ implies that $f_i^{(m)} \in \mathcal{F}(\mathcal{A}_i)$ as well.  Thus if the inequality holds for step functions, we have
    \[
        \nu_*\left( \square f \right) \geq \nu\left( \square f^{(m)} \right) \geq \mathcal{M}_\alpha^{(t)}\left(\| f^{(m)} \|_1\right),
    \]
    where the first inequality follows from the monotonicity of the sup-convolution, Definition~\ref{def: sup convolution} \eqref{item: monotonicity of sup conv}, and the fact that $f \geq f^{(m)}$.  By the monotone convergence theorem, $\|f^{(m)}\|_p$ increases to $\|f\|_p$ and by the continuity of $\mathcal{M}_\alpha^{(t)}(\cdot)$, we get
    \[
        \mathcal{M}_\alpha^{(t)} (\|f\|_1) = \lim_{m \to \infty}\mathcal{M}_\alpha^{(t)} (\|f^{(m)}\|_1) \leq \nu( \square f).
    \]
    Thus we may assume that the $f_i$ take only finitely many values and prove the result through induction on the total number of non-zero values taken by the $f_i$.  That is, if $N_i$ denotes the number of distinct non-zero values taken by $f_i$, we denote $N = \sum_{i=1}^n N_i$.  The base case follows by hypothesis. Let us assume $N > n$, that the $\mathcal{F}(\mathcal{A})$-functional inequality holds for step functions who take strictly less than $N$ total values, and that all $f_i$ are non-zero step functions.
    For each $i=1,\dots,n$, we write 
    \[
        f_i = \sum_{j =1}^{N_i} a_{ji} \mathbbm{1}_{A_{ji}},
    \]
    with $a_{ji} > 0$, $A_{ji} \in \mathcal{A}_i$ and $A_{ji} \supseteq A_{(j + 1)i } $, with $\sum_{i=1}^n N_i = N$.  Define
    \[
        f_1^- = a_{1 1} \mathbbm{1}_{A_{11}} \hbox{ and } \ f_1^+ = \sum_{j = 2}^{N_1} a_{j1} \mathbbm{1}_{A_{j1}}.
    \]
    Note that $f_1^-$ and $f_1^+$ individually take strictly fewer values than $f_1=f_1^-+f_1^+$.  Define 
    \[
        \lambda  = \frac{ \|f^-_1\|_{1}}{\|f_1\|_{1}},
    \]
    so that $1-\lambda =  \frac{ \|f^+_1\|_{1}}{\|f_1\|_{1}}$.  Note that $\lambda \in (0,1)$.
    Let
    \[
        f_i^- = \min\{ f_i, s_i\} \ \hbox{ and } \ f_i^+ = \max\{ f_i - s_i, 0 \}
    \]
    where $s_i$ is implicitly and uniquely defined through the equation
    \[
        \frac{\|f_i^-\|_{1}}{\|f_i\|_{1}} = \lambda
    \]
    because $s \mapsto \mu_i( \min\{f_i, s\} )$ is continuous and strictly increasing on $(0, \|f_i\|_\infty)$. Note that $f_i=f_i^-+f_i^+$ and that 
    $1-\lambda=\frac{\|f_i^+\|_{1}}{\|f_i\|_{1}}$, $i=1,2,\dots,n$.
    
    Set $f^+ = (f_1^+,\dots, f_n^+)$ and $f^-=(f_1^-,\dots,f_n^-)$. By Lemma \ref{lem: f- and f+ satisfy supconv condition} and Definition \ref{def: sup convolution} Item \eqref{item: super additivity on sliced functions}
    \[
        \square f \geq \square f^- + \square f^+.
    \]
    Integrating this inequality and applying the inductive hypothesis,
    \begin{align*}
        \nu( \square f) 
            &\geq
                \nu( \square f^-) + \nu(\square f^+)
                    \\
            &\geq
                \mathcal{M}_\alpha^{(t)}(\|f^-\|_{1}) + \mathcal{M}_\alpha^{(t)}(\|f^+\|_{1})
                    \\
            &=
                \mathcal{M}_\alpha^{(t)}( \lambda   \|f \|_{1}) + \mathcal{M}_\alpha^{(t)}((1-\lambda)  \|f \|_{1})
                    \\
            &=
                \mathcal{M}_\alpha^{(t)}( \|f \|_{1}),
    \end{align*}
as desired. 
Let now $p_i>1$ for all $i\in I\subset\{1,\ldots,n\}$ and $q_i$ be the dual exponent of $p_i,$ for $i\in I$. We may write
\[
\mathcal{M}_\alpha^{(t)}(\|f_i\|_{p_i})=\sup_{\|g_i\|_{q_i}=1, i\in I} \mathcal{M}_\alpha^{(t)}\left(\left(\int f_i g_i d\mu_i\right)_{i\in I}, \left(\|f_i\|_1\right)_{i\in I^c}\right).
\]
Thus (\ref{eq: abstract sup theorem conclusion}) is equivalent to the fact that for each $g_i\in L^{q_i}(\mu_i)$, with $\|g_i\|_{q_i}=1$ it holds that
\begin{equation} \label{eq: equivalent to abstract sup theorem conclusion}
\nu( \square f) \geq \mathcal{M}_\alpha^{(t)}\left(\left(\int f_i g_i d\mu_i\right)_{i\in I}, \left(\|f_i\|_1\right)_{i\in I^c}\right).
\end{equation}
For each $i\in I$, let $g_i\in L^{q_i}(\mu_i)$, with $\|g_i\|_{q_i}=1$ and consider the measures $\tilde{\mu}_i:=g_i\mu_i$, when $i\in I$. Then (\ref{eq: abstract sup theorem assumption}) implies that
\[
\nu( \square a) \geq \mathcal{M}_\alpha^{(t)}\left(\left(\|a_i\|_{L^1(\tilde{\mu}_i)}\right)_{i\in I}, \left(\|a_i\|_1\right)_{i\in I^c}\right),
\]
for each $a  \coloneqq (a_1 \mathbbm{1}_{A_1}, \dots, a_n \mathbbm{1}_{A_n})$ for $a_i \in (0,\infty)$ and $A_i \in \mathcal{A}_i$. Since for each $i\in I$, $f_i\in L^{1}(\tilde{\mu}_i)$ and for $i\in I^c$, $f_i\in L^1(\mu_i)$, the previous case where all $p_i$'s were equal to $1$ gives (\ref{eq: equivalent to abstract sup theorem conclusion}) and thus concludes the proof for $p\in [1,\infty)^n$.

\end{proof}

Concerning operations in line with our concentrations are the articles \cite{GHW13,GK18}. In \cite{GHW13} Gardner, Hug, and Weil conducted a  systemic study of operations between sets in $\R^n$ with a particular emphasis on compact convex sets in $\R^n$. Among these operations are the Minkowski sum, and $L_p$-Minkowski sum (see Section~\ref{sec:lp}), which we consider here. Additionally, the article \cite{GK18} of Gardner and Kiderlin concerns what they called ``point-wise, homogeneous, monotonic, and associative'' between functions defined on $\R^n$, and they characterize such operations as representable by $40$ explicitly given types.  In particular, \cite[Section~11]{GK18} concerns a characterization of the supremal convolution appearing in \eqref{eq:operationexample} on the class of logarithmically concave functions on $\R^n$, denoted $\text{LC}(\R^n)$. They show, \cite[Theorem~11.1]{GK18}, that any operation $\star:\text{LC}(\R^n) \times \text{LC}(\R^n) \to \text{LC}(\R^n)$ verifying certain properties, in addition to the above ones, must arise as a sup-convolution.  Though most of our focus in this work is in Euclidean space, the framework for our ``generalized sup-convolution'' is much more flexible, approaching inequalities in measure spaces, both geometric and functional.  An analogous classification of the seems beyond reach currently.


\section{Dimensional Functional Brunn-Minkowksi inequalities in measure spaces} \label{sec:mesures}




In this section we use Theorem \ref{thm: abstract sup convolution theorem} to obtain functional versions of dimensional Brunn-Minkowski inequalities. We will start with the proof of Theorem \ref{thm: rotationally invariant BBL for unimodal}, for which the main result of \cite{cordero2023improved} is required.

\begin{alphatheorem}[Cordero-Erausquin \& Rotem, \cite{cordero2023improved}] \label{thm: Cordero Rotem}
    Let $w:[0,\infty)\to (-\infty,\infty]$ be an increasing function such that $\R^n\ni s \mapsto w(e^s)$ is convex, and $\mu$ the measure on $\mathbb{R}^n$ with density $\frac{d\mu}{dx} = e^{-w(|x|)}$. Then for symmetric convex bodies $A$ and $B$
    \[
        \mu((1-t)A + tB ) \geq \mathcal{M}_{\frac 1 n}^{(t)}(\mu(A), \mu(B)). 
    \]
\end{alphatheorem}

\begin{proof}[Proof of Theorem \ref{thm: rotationally invariant BBL for unimodal}]
    We prove the result first when $\alpha= - \frac 1 n$.  Take $\mathcal{A} = (\mathcal{A}_1, \mathcal{A}_2)$ for $\mathcal{A}_i$ the cone of all symmetric convex bodies, and $\Phi: \mathbb{R}^n \times \mathbb{R}^n \to \mathbb{R}^n$ to be $\Phi(x,y) = (1-t)x + ty$ and $C(z) = \{(x,y): \Phi(x,y) = z\}$. According to Theorem \ref{thm: everything is sup convolution},
    \[
        f\square g(z) \coloneqq  \sup_{(1-t)x +ty = z} \mathcal{M}^{(t)}_{-\frac 1 n}(f(x),g(y))
    \]
    is a sup-convolution. Thus, in view of Theorem~\ref{thm: abstract sup convolution theorem},  our result holds if it holds for scaled indicator functions. Suppose that $f = a \mathbbm{1}_A$ and $g = b \mathbbm{1}_B$ for $A$ and $B$ symmetric convex bodies. For the remainder of the proof, we will always assume that $f$ and $g$ are of the above form. Observe that
    \begin{align*}
        \mu(f \square g) 
            &=
                \mathcal{M}_{- \frac 1 n}^{(t)}(a,b) \mu((1-t)A + tB)
                    \\
            &\geq
                \mathcal{M}_{- \frac 1 n}^{(t)}(a,b) \mathcal{M}_{\frac 1 n}^{(t)}(\mu(A),\mu(B))
                    \\
            &\geq \min\{ a \mu(A), b \mu(B) \},
    \end{align*}
    completing the proof for $\alpha = - \frac 1 n$. 
    
    Assume next that $\alpha > - \frac 1n$. Then, by applying Holder's inequality \eqref{eq:holder-mean}, we obtain
\begin{align*}
       h((1-t)x + ty) &\geq \mathcal{M}_\alpha^{(t)}(f(x),g(y))\\
       &= \mathcal{M}_\alpha^{(t)} \left( \frac{f(x)}{\mu(f)} \mu(f), \frac{g(y)}{\mu(g)} \mu(g) \right)\\
       &\geq \mathcal{M}_{- \frac 1 n}^{(t)} \left( \frac{f(x)}{\mu(f)}, \frac{g(y)}{\mu(g)} \right) \mathcal{M}_r^{(t)}( \mu(f), \mu(g)).
\end{align*}
    Applying the $\alpha = - \frac 1 n$ result to this inequality completes the proof of the general case. For the converse we apply the functional inequality to the characteristic functions of two symmetric convex bodies.
\end{proof}

In \cite{CE2025} a different functional formulation of the dimensional Brunn--Minkowski inequality was given and reads as follows (with the notation used so far):

\begin{alphatheorem}[Cordero-Erausquin \& Eskenazis, \cite{CE2025}]
\label{thm:dimensionalconcavitygeneral}
    Let $\mu$ be the rotationally invariant measure on $\R^n$ as previously. Let $\Omega\subset\R^{n+1}$ be convex, with the property that for each $t\in \R$, $\Omega_t:=\{x\in \R^n: (t,x)\in \Omega\}$ is symmetric, and $\Phi:\Omega\to [0,\infty)$ be concave, with $\Phi(t,\cdot)$ even for each $t\in \R$. Then 
    \begin{equation}
        \phi(t):= \left( \int_{\Omega_t} \Phi(t,x)^{\frac{1}{\alpha}}\ d\mu(x)   \right)^{\frac{\alpha}{1+n\alpha}},
    \end{equation}
    is concave on its support for each $\alpha >0$, whenever the integral converges.
\end{alphatheorem}

Further, in \cite{CE2025} it was asked whether the function $\phi$ is \textit{convex}, whenever $\alpha \in (-\frac{1}{n},0)$ and $\Phi$ is convex. Using Theorem \ref{thm: rotationally invariant BBL for unimodal} we can answer positively this question for the class of measures considered so far and offer another proof of \cite[Theorem 1]{CE2025}.

\begin{cor} \label{cor: general cordero--eskenazis}
    Let $\mu$ be the measure on $\mathbb{R}^n$, with density $e^{-w(|x|)}$, where $w:[0,\infty)\to (-\infty,\infty]$ is increasing and $\mathbb{R}\ni t \mapsto w(e^t)$ is convex. Let $\Omega\subset\R^{n+1}$ be convex, with the property that for each $t\in \R$, $\Omega_t:=\{x\in \R^n: (t,x)\in \Omega\}$ is symmetric, and $\Phi:\Omega\to [0,\infty)$ a function with $\Phi(t,\cdot)$ even for each $t\in \R$. Consider the function
    \begin{equation}
         \phi(t):= \left( \int_{\Omega_t} \Phi(t,x)^{\frac{1}{\alpha}}\ d\mu(x)   \right)^{\frac{\alpha}{1+n\alpha}}.
    \end{equation} 
    \begin{enumerate}
        \item If $\alpha>0$ and $\Phi$ is concave, then $\phi$ is concave on its support, whenever the integral converges.

        \item If $\alpha \in (-\frac{1}{n},0)$ and $\Phi$ is convex, then $\phi$ is convex on its support, whenever the integral converges.
        
    \end{enumerate}
\end{cor}

\begin{proof}
Consider the functions:
\[
h(z):=\Phi(\lambda t+(1-\lambda)s,z)^{\frac{1}{\alpha}}\mathbbm{1}_{\Omega_{\lambda t+(1-\lambda)s}}(z),\ f(x):=\Phi( t,x)^{\frac{1}{\alpha}}\mathbbm{1}_{\Omega_{t}}(x) \text{  and  } g(y):= \Phi( s,y)^{\frac{1}{\alpha}}\mathbbm{1}_{\Omega_{s}}(y).
\]
Then we have that 
\[
h(\lambda x+(1-\lambda)y)\geq \mathcal{M}_\alpha^{(\lambda)}(f(x),g(y))
\]
and thus Theorem~\ref{thm: rotationally invariant BBL for unimodal} gives that $\phi$ has the desired property in both cases.
\end{proof}

Aishwarya and Rotem \cite{aishwarya2023new} recently used a new approach for dimensional Brunn-Minkowski inequalities using tools from optimal transport and information theory. Their entropic method produced the following geometric result.

\begin{alphatheorem}[Aishwarya \& Rotem, \cite{aishwarya2023new}]\label{thm: Aishwarya Rotem}
    Let $V: \mathbb{R}^n \to [0,\infty)$ a $p$-homogeneous convex function for $1 < p < \infty$, and let $\nu$ be a measure with density $\frac{d\nu}{dx}$ proportional to $e^{-V}$.  Let $K_0$ and $K_1 \subseteq \mathbb{R}^n$ be star bodies.  Then for $0 < t < 1$,
    \[
        \nu((1-t) K_0 + t K_1) \geq \mathcal{M}^{(t)}_{\frac{p-1}{pn}}(\nu(K_0),\nu(K_1))
    \]
\end{alphatheorem}

Here $K$ is a star body if $K$ is Borel and satisfies $\lambda K \subseteq K$ or $\lambda \in [0,1]$ and $V$ a $p$-homogeneous function means that $V(\lambda x) = \lambda^p V(x)$ for $\lambda \geq  0$ and $x \in \mathbb{R}^n$.  We consider a function $f: \mathbb{R}^n \to \mathbb{R}$ to be {\it star-unimodal}, if $t \in \mathbb{R}$ implies $\{f \geq t\}$ is a star body. 

An application of Theorem \ref{thm: abstract sup convolution theorem} for this class of sets, similar to the above proof, gives the following functional analogue.

\begin{thm} Let $V \colon \R^n \to [0,\infty)$ be $s$-homogeneous convex function for some $1< p < \infty$, and let $\nu$ be the measure with density $\frac{d\nu}{dx}=e^{-V}$, and let $t \in (0,1)$ and $\alpha \geq -\frac{p-1}{pn}$. If  $h,f,g: \mathbb{R}^n \to [0,\infty)$ with $f,g$ star-unimodal, is a triple of functions satisfying
    \[
        h((1-t)x + ty) \geq \mathcal{M}_\alpha^{(t)}(f(x),g(y)),
    \]
    then
    \[
        \nu(h) \geq \mathcal{M}_r^{(t)} (\nu(f), \nu(g)),
    \]
    where $r = \frac{\alpha(p-1)}{\alpha p n + p-1}$.
\end{thm}

\begin{proof}
The methods of Theorem \ref{thm: rotationally invariant BBL for unimodal} can be applied to $\mathcal{A}_i$ the collection of star bodies and applying Theorem~\ref{thm: Aishwarya Rotem}.
\end{proof}

When restricted to radially decreasing functions and $\alpha>0$, this becomes \cite[Theorem~1.8]{AD25}.

\section{Functional $L_p$-Brunn-Minkowski inequalities} \label{sec:lp}

In this section we investigate how our methods can be used to provide functional inequalities in the $L_p$-Brunn-Minkowski theory. In particular, we develop a functional equivalent of the celebrated $\log$-Brunn-Minkowski conjecture from \cite{BLYZ12}. 

This theory has been pivotal in many advancements in convex geometry and affine differential geometry over the last thirty years. The literature is extensive, so we direct the reader to the articles \cite{Firey62,LE93,LE96,LZ97,LYZ00, LYZ02, LYZ04, LYZ04_2, LYZ05,LYZ05_2, LYZ06}, which demonstrate the versatility of this subject and its connection to other fields.  All of the above concern the case of $p \geq 1$. 

On the other hand, very recently the study of the case $p \in [0,1]$ has gained much momentum; see the survey \cite{Bor23} of B\"or\"oczky for a detailed history, and related results and conjectures, and their connections to other fields. In particular, we would like to direct the reader's attention to the works \cite{BLYZ12,Marsiglietti,LNMZ,CS15,CS16, CLM17,KM22,CHLL20}.

We remind the reader that we take a convex body to be a compact convex set with non-empty interior.
For convex bodies containing the origin in their interior, $K_0$ and $K_1$, $t \in (0,1)$ and $p \in [0,\infty)$ one defines the $L_p$-Minkowski convex combination of $K_0$ with $K_1$, denoted $(1-t) \cdot K_0 +_p t \cdot K_1$,
via the W\"ulff shape of $p$-th average of the support functions of the $K_i$.  More explicitly, for a set $A \subseteq \mathbb{R}^n$ define (the support function)
$
    h_A: \mathbb{R}^n \to \mathbb{R}
$ by
\begin{equation} \label{eq:support}
    h_A(z) \coloneqq \sup_{x \in A} \langle x, z \rangle. 
\end{equation}
Let $(1-t) \cdot K_0 +_p t \cdot K_1$ be defined as
\begin{equation}\label{eq:lpsum}
(1-t) \cdot K_0 +_p t \cdot K_1:=    \bigcap_{\theta \in \mathbb{S}^{n-1} }\left \{x \in \mathbb{R}^n : \langle x, \theta \rangle \leq \mathcal{M}_p^{(t)}\left(h_{K_0}(\theta), h_{K_1}(\theta) \right) \right\}.
\end{equation}

Note that when $p=1$ the $L_p$-Minkowski sum amounts to the classical Minkowski addition 
$(1-t)K_0+tK_1$.

For $p=0$ we shorten the notation and write 
$$
K_0^{1-t}K_1^t \coloneqq (1-t) \cdot K_0 +_0 t \cdot K_1.
$$

In the case $p > 1$, and $q = \frac{p}{p-1}$, this definition extends to Borel sets in \cite{LYZ12} in the following way: given Borel measurable sets $A,B \subset \R^n$ and $t \in (0,1)$, define
\begin{align*}
(1-t) A +_p t B &= \bigcup_{\lambda \in [0,1]} \left\{ (1-t)^{\frac{1}{p}}(1-\lambda)^{\frac{1}{q}} A + t^{\frac{1}{p}}\lambda^{\frac{1}{q}}B\right\}\\
&=\left\{z : \exists (a,b,\lambda) \in A \times B \times [0,1] \text{ with }\ z = (1-t)^{\frac 1 p} (1-\lambda)^{\frac 1 q} a + t^{\frac 1 p} \lambda^{\frac 1 q} b \right\}.
\end{align*}
As proved in \cite[Lemma 1-2]{LYZ12} this interpretation of the $L_p$-Minkowski sum  agrees with \eqref{eq:lpsum} when $p > 1$ 
 and when $A,B$ are convex bodies containing the origin in their interiors. Therein, the authors established the following general form of the $L_p$-Brunn-Minkowski inequality for Borel sets for $p >1$:
\begin{equation}\label{e:LpBMBorel}
|(1-t)A+_p tB| \geq \mathcal{M}_{\frac p n}^{(t)} (|A|,|B|). 
\end{equation}
Since $(1-t) A + t B \subseteq (1-t) A  +_p t B$ for $p >1$ (a non-trivial inclusion proved in \cite[Lemma 3]{LYZ12}), inequality \eqref{e:LpBMBorel} follows from the classical Brunn-Minkowski inequality \eqref{eq: BMI} and a homogeneity argument. 

For $p \in [0,1)$, the following conjecture -- that has deeply influenced the research in convex geometry in recent years -- was put forth by B\"or\"oczky, Lutwak, Yang and Zhang in \cite{BLYZ12}: 

\begin{conj}[$L_p$-Brunn-Minkowski Conjecture \cite{BLYZ12}] \label{c:lpbm} Let $0 \leq p < 1$. Given any pair of origin symmetric convex bodies $K,L \subset \R^n$ and $t \in (0,1)$, does it necessarily follow that 
\[
|(1-t)K+_p tL| \geq \mathcal{M}_{\frac p n}^{(t)} (|K|,|L|)?
\]
    
\end{conj}

It should be noted that the inequality appearing in Conjecture~\ref{c:lpbm} dramatically improves the classical Brunn-Minkowski inequality.  The case $p =0$ is known as the \emph{$\log$-Brunn-Minkowski conjecture}. Work towards this conjecture and related problems are rapidly developing and we do not attempt to cover all of the related papers and results.  Instead we choose to focus on a few important cases for which this conjecture was verified: for $n =2$ in \cite{BLYZ12}, locally near a Euclidean ball in \cite{CLM17}, and for unconditional convex bodies in \cite{CS15} when $p=0$ and for $0 < p < 1$ in \cite{Marsiglietti15}. Further it was verified in \cite{CS16} that the $\log$-Brunn-Minkowski inequality for volume implies the same inequality with volume replaced by any even $\log$-concave measure on $\R^n$; a similar result for $0 <p <1$ was established in \cite{LNMZ}. In \cite{KM22}, a general treatment of local $L_p$-Brunn-Minkowski inequalities were treated, and in \cite{CHLL20}, Conjecture~\ref{c:lpbm} was verified for to hold for all $p \in (p_0,1)$ for some  $p_0$ depending on the dimension.

In the following subsections we will demonstrate how our methods can be used to prove functional versions of geometric inequalities appearing in the $L_p$-Brunn-Minkowski theory. We also provide a functional equivalent of the geometric inequality appearing in Conjecture~\ref{c:lpbm}.



\subsection{The case $p \in [0,1)$}

Here we move to the case $p \in [0,1)$. We begin with a definition. 

\begin{defn} Let $0\leq p \leq 1$. Let $\mathcal{A}$ be the cone of all origin symmetric convex bodies in $\mathbb{R}^n$,  so that $\mathcal{F}(\mathcal{A})$ is the collection of all even  unimodal functions. For $f = (f_0,f_1)$, with $f_i \in \mathcal{F}(\mathcal{A})$, define the sup-convolution
    \begin{align} \label{eq: Lp BMI supconv}
        \square_p f (z) \coloneqq f_0 \square_{p} f_1(z) \coloneqq \sup_{(x_0,x_1) \in C(z;f)} \mathcal{M}_{- \frac p n}^{(t)}(f_0(x_0), f_1(x_1)) ,
    \end{align}
    where
    \[
        C_p(z;f) \coloneqq \left\{ (x_0,x_1) : z \in (1-t) \cdot \{f_0 \geq f_0(x_0) \} +_p t \cdot \{ f_1 \geq f_1(x_1) \} \right\}.
    \]
    \end{defn}

The next theorem assets that $\square_p$ is indeed a sup-convolution, in the sense of Definition \ref{def: sup convolution}, when $p \in [0,1]$.
    
    \begin{thm} \label{thm: Lp Borell BL}
    Let $p \in [0,1]$. 
        For $\mathcal{A}$ the cone of origin symmetric convex bodies, let $f=(f_0,f_1)$, with $f_i \in \mathcal{F}(\mathcal{A})$. Then the map $\square_p f$ defined in \eqref{eq: Lp BMI supconv} is a sup-convolution.
    \end{thm}

\begin{proof}
    We first prove monotonicity. Let $f=(f_0,f_1), g=(g_0,g_1)$ with $f_i,g_i \in \mathcal{F}(\mathcal{A})$ satisfying $f_i \leq g_i$. Note that if $f_i$ is the zero function, $\square_p f(z)$ is identically zero and there is nothing to prove. Now suppose that $f_i, g_i \in \mathcal{F}(\mathcal{A})$, are such that $f_i \leq g_i$ for all $i$.  Fix $x = (x_0,x_1) \in  C_p(z;f)$ such that $f_0(x_0) f_1(x_1) >0$. By the condition $f_i \leq g_i$, we have 
    \[
        \{ f_i \geq f_i(x_i) \} \subseteq \{ g_i \geq f_i(x_i) \}.
    \]
    Since $\{ g_i \geq f_i(x_i) \} \in \mathcal{A}$, it is necessarily compact, and hence there exists $y_i \in \{ g_i \geq f_i(x_i) \}$ such that $g_i(y_i) \leq g_i(x)$ for all $x$ such that $g_i(x) \geq f_i(x_i)$.  Thus
    \[
        \{ g_i \geq g_i(y_i) \} = \{ g_i \geq f_i(x_i) \} \supseteq \{f_i \geq f_i(x_i) \}.
    \]
    Thus, $h_{\{ g_i \geq g_i(y_i) \} } \geq h_{\{ f_i \geq f_i(x_i) \} }$, and so we have that 
    \[
        z \in (1-t) \cdot \{ f_0 \geq f_0(x_0) \} +_p t \cdot \{f_1 \geq f_1(x_1) \} \subseteq (1-t) \cdot \{ g_0 \geq g_0(y_0) \} +_p t \cdot \{g_1 \geq g_1(y_1) \}.
    \]
    Consequently, $(y_0,y_1) \in C_p(z;g)$, and hence 
    \[
        \mathcal{M}_{- \frac p n}^{(t)}(f_0(x_0), f_1(x_1)) \leq \mathcal{M}_{- \frac p n}^{(t)}(g_0(y_0), g_1(y_1)) \leq \sup_{w \in C(z;g)} \mathcal{M}_{- \frac p n}^{(t)}(g_0(w_0), g_1(w_1)) = \square_p g(z).
    \]
    Taking the sup over all $x \in C_p(z,f)$ completes the proof of monotonicity.

{ Now
suppose that $f_i$ and $g_i$ are step functions and are such that $g_i(x_i) > 0$ implies that $f_i(x_i) \geq f_i(y_i)$ for all $y_i \in \mathbb{R}^n$.  We wish to show that $\square_p(f + g) \geq \square_p f + \square_p g$. If $\square_p g(z) = 0$, then $\square_p (f + g)(z) \geq \square_p f(z)$ by monotonicity.  Thus we can assume that $\square_p g(z) >0$. 
 Given $\varepsilon >0$  there exists $x = (x_0,x_1) \in C(z;g)$ such that
 \[
 \mathcal{M}_{- \frac p n}^{(t)}(g_0(x_0), g_1(x_1)) > \square_p g(z) - \varepsilon
\]
and $g_i(x_i)>0$ for all $i$.
 Since  $f_i(x_i) \geq f_i(y_i)$ for all $y_i$ (by our assumption), we get
 \[
\mathcal{M}_{- \frac p n}^{(t)}(f_0(x_0), f_1(x_1)) \geq \mathcal{M}_{- \frac p n}^{(t)}(f_0(y_0), f_1(y_1)).
 \]
 Taking the suprema of the right hand side of the above equation over $y \in C_p(z;f)$ gives
 \[
\mathcal{M}_{- \frac p n}^{(t)}(f_0(x_0),f_1(x_1)) \geq \square_p f(z).
 \]
 However, since $x \in C_p(z;g)$ we have by the hypothesis on $f$ and $g$, 
 \[
    \{ g_i \geq g_i(x_i)\} \subseteq \{ g_i > 0 \} \subseteq \{f_i \geq f_i(x_i)\},
 \]
 so that $$z \in (1-t) \cdot \{ g_0 \geq g_0(x_0) \} +_p t \cdot \{g_1 \geq g_1(x_1) \} \subseteq (1-t) \cdot \{f_0 \geq f_0(x_0) \} +_p t \cdot \{f_1 \geq f_1(x_1) \},$$
 whence $x \in C_p(z;f)$.  
 Thus,
 \[
\mathcal{M}_{- \frac p n}^{(t)}(f_0(x_0), f_1(x_1)) =  \square_p f(z).
 \]
 Applying the reverse triangle inequality, $\mathcal{M}_{- \frac p n}^{(t)}(u_0,u_1) + \mathcal{M}_{- \frac p n}^{(t)}(v_0,v_1) \leq \mathcal{M}_{- \frac p n}^{(t)}(u_0 + v_0, u_1 + v_1)$, we have
\begin{align*}
    \square_p f(z) + \square_p g(z) &\leq \mathcal{M}_{- \frac p n}^{(t)}(f_0(x_0), f_1(x_1)) + \mathcal{M}_{- \frac p n}^{(t)}(g_0(x_0), g_1(x_1)) + \varepsilon\\
    &\leq \mathcal{M}_{- \frac p n}^{(t)}((f_0 + g_0)(x_0), (f_1+g_1)(x_1)) + \varepsilon.
\end{align*}
 To complete the proof, it remains to see that $x \in C_p(z; f+g)$ so that the suprema can be taken, and then $\varepsilon \to 0$.  But this is immediate, as $\{g_i \geq g_i(x_i)\} \subseteq \{f_i \geq f_i(x_i) \} = \{f_i = f_i(x_i)\}$, 
 we infer that $\{g_i \geq g_i(x_i)\} \subseteq \{f_i + g_i \geq (f_i + g_i)(x_i) \}$; and we have
 \begin{align*}
    z \in (1-t) \cdot \{g_0 \geq g_0(x_0)\} &+_p t \cdot \{g_1 \geq g_1(x_1) \} \\
    &\subseteq (1-t) \cdot \{(f_0 + g_0)  \geq (f_0 + g_0)(x_0) \} +_p t \cdot \{ (f_1 + g_1) \geq (f_1 + g_1) (x_1) \}.
 \end{align*}

For $f_0 = \mathbbm{1}_{A}$ and $f_1 = \mathbbm{1}_B$ we claim that
\begin{align*}
    \square_p f (z) = \mathbbm{1}_{(1-t) \cdot A +_p t \cdot B}(z).
\end{align*}
$\square_p f(z) =1$ if there exists $a \in A$ and $b \in B$ such that $(a,b) \in C_p(z;f)$ and is $0$ otherwise, but $(a,b) \in C_p(z,f)$ is equivalent to 
\[
    z \in (1-t) \cdot \{ f_0 \geq f_0(a)\} +_p t \cdot \{ f_1 \geq f_1(b) \} = (1-t) \cdot \{ f_0 \geq 1\} +_p t \cdot \{ f_1 \geq 1 \} = (1-t) \cdot A +_p t \cdot B.
\]
Thus, the measurability of $\square_p f$ in this case follows from the fact that $(1-t) \cdot A +_p t \cdot B$ is a convex set.
}
\end{proof}

The next statement establishes an equivalent functional version of the
$L_p$-Brunn-Minkowski inequality.

\begin{thm} \label{thm:equivalenoflpBM}
    For $p \in [0,1)$, the $L_p$-Brunn-Minkowski inequality
    \[
        |(1-t) \cdot A +_p t \cdot B | \geq \mathcal{M}_{\frac p n}^{(t)} (|A|, |B|)
    \]
    holding for all origin symmetric convex bodies $A$ and $B$  in $\R^n$ is equivalent to the following functional inequality: given $\alpha \geq - \frac p n$, $t \in (0,1)$, and a triple of functions $h,f,g \colon \R^n \to [0,\infty)$, with $f,g$ even and unimodal, satisfying 
     \[
        h(z) \geq \mathcal{M}_\alpha^{(t)}(f(x),g(y)),
    \]
    where $$(x,y) \in C_p(z;f,g) \coloneqq \{ (u,v) : z \in (1-t)\cdot \{f \geq f(u) \} +_p t \cdot \{ g \geq g(v) \},$$
    then
    \[
        \int_{\R^n} h \geq \mathcal{M}_r^{(t)} \left( \int_{\R^n} f , \int_{\R^n} g \right)
    \]
    for $r = \frac{\alpha p}{p + \alpha n}$.
\end{thm}

\begin{proof}
    If the set theoretic result holds, by Theorem \ref{thm: Lp Borell BL} the operation $f \square_{p} g$ is a sup-convolution, so by Theorem \ref{thm: abstract sup convolution theorem} is suffices to prove the functional inequality  when $f = a \mathbbm{1}_A$ and $g = b \mathbbm{1}_B$, where $A$ and $B$ are origin symmetric convex bodies, and $a,b >0$.  In this case
    \[
        h(z) \geq f \square_p  g (z) = \mathcal{M}_{-\frac p n}^{(t)} (a,b) \mathbbm{1}_{(1-t) \cdot A +_p t \cdot B}.
    \]
    Integrating both side and applying H\"older's inequality \eqref{eq:holder-mean} and the $L_p$-Brunn-Minkowski inequality, we have
    \begin{align*}
        \int_{\R^n} h \geq \int_{\R^n} f \square_p g 
        &=
            \mathcal{M}_{-\frac p n}^{(t)}(a,b) |(1-t) \cdot A +_p t \cdot B|
                \\
        &\geq
            \mathcal{M}_{-\frac p n}^{(t)}(a,b) \mathcal{M}_{\frac p n}^{(t)}(|A|, |B|)
                \\
        &=
            \mathcal{M}_{- \infty}^{(t)}( a |A|, b |B|).
    \end{align*}
    This yields the functional inequality when $\alpha = - \frac p n$.  When $\alpha > - \frac p n$, we have by assumption, and then by H\"older's inequality \eqref{eq:holder-mean}, that 
    \begin{align*}
        h(z) 
            &\geq 
                \sup_{(x,y) \in C(z;f,g)} \mathcal{M}_{\alpha}^{(t)} \left( \frac{f(x)\int_{\R^n} f }{\int_{\R^n} f } , \frac{g(y) \int_{\R^n} g }{\int_{\R^n} g }\right)
                    \\
            &\geq
                \sup_{(x,y) \in C(z;f,g)} \mathcal{M}_{-\frac p n}^{(t)}\left( \frac{f(x)}{\int_{\R^n} f }, \frac{g(y)}{\int_{\R^n} g} \right) \mathcal{M}_r^{(t)}\left( \int_{\R^n} f , \int_{\R^n} g \right).
    \end{align*}
    Now observing the scaling invariance of $C(z;f,g)$, that is that $C(z; \lambda f, t g) = C(z; f, g)$ for $\lambda, t > 0$ we see that the above inequality can be written as
    \[
        h(z) \geq \frac{f}{\int_{\R^n} f} \square_p \frac{g}{\int_{\R^n} g }(z)\mathcal{M}_r^{(t)}\left( \int_{\R^n} f, \int_{\R^n} g \right).
    \]
    Integrating, we conclude that
\begin{align*}
        \int_{\R^n} h &\geq \mathcal{M}_r^{(t)}\left( \int_{\R^n} f, \int_{\R^n} g \right) \int_{\R^n} \left( \frac{f}{\int_{\R^n} f} \square_p \frac{g}{\int_{\R^n} g }  \right)\\
        &\geq \mathcal{M}_r^{(t)}\left( \int_{\R^n} f , \int_{\R^n} g \right) \mathcal{M}_{-\infty} \left( \int_{\R^n} \frac {f}{\int_{\R^n} f} , \int_{\R^n} \frac{g}{\int_{\R^n} g } \right)\\
        &= \mathcal{M}_r^{(t)} \left( \int_{\R^n} f , \int_{\R^n} g \right).
\end{align*}
    For the converse, if there exists $\alpha \geq - \frac p n$ such that the functional inequality holds, with $A$ and $B$ origin symmetric convex bodies and let $f = a \mathbbm{1}_A$, $g = b \mathbbm{1}_B$.  Applying the functional inequality gives
    \[
        \mathcal{M}_\alpha^{(t)}(a,b) |(1-t) \cdot A +_p t \cdot B| \geq \mathcal{M}_r^{(t)}(a |A|, b |B|) = \mathcal{M}_\alpha^{(t)}(a,b) \mathcal{M}_{\frac p n}^{(t)}(|A|,|B|),
    \]
    where the last equality gives the result by taking $a = |A|^{\frac{p}{\alpha n}}$ $ b = |B|^{\frac{p}{\alpha n }}$, chosen to satisfy equality in H\"older.
\end{proof}


Again, in cases where geometric inequalities are already verified, with Theorem~\ref{thm:equivalenoflpBM} in hand, one obtains new $L_p$-Borell-Brascamp-Lieb type inequalities for $p <1$. For example we use the following result from \cite{CHLL20}.

\begin{alphatheorem}[Chen, Huang, Li, \& Liu, \cite{CHLL20}] \label{thm:partialLpBM} There exists a $p_0 \in (0,1)$, depending on the dimension, such that for all $p \in (p_0,1)$  and for any pair of origin symmetric convex bodies $K,L \subset \R^n$ and any $t \in [0,1]$, it holds that 
\[
|(1-t) \cdot K +_{p} t \cdot L| \geq \mathcal{M}_{\frac{p}{n} }^{(t)}(|K|,|L|). 
\]
\end{alphatheorem}

\begin{thm} Let $p_0$ be as in Theorem~\ref{thm:partialLpBM} and let $p \in (p_0,1)$. Let $\alpha \geq - \frac{p}{n}$, and $h,f,g \colon \R^n \to [0,\infty)$ be a triple of functions, with $f,g$ being even and unimodal, such that 
\[
h(z) \geq \mathcal{M}_\alpha^{(t)}(f(x),g(y))
\]
for $$(x,y) \in C_p(z;f,g) \coloneqq \{ (u,v) : z \in (1-t)\cdot \{f \geq f(u) \} +_p t \cdot \{ g \geq g(v) \}.$$
Then the following $L_p$-Borell-Brascamp-Lieb inequality holds: 
\[
\int_{\R^n}h \geq \mathcal{M}^{(t)}_r\left(\int_{\R^n}f, \int_{\R^n}g \right),
\]
where $r = \frac{p\alpha}{p+\alpha n}$. 
\end{thm}

Next, we continue with the case of $p=0$. For starters, we take note in the next results that the $\log$-Brunn-Minkowski conjecture enjoys the following self-improvement results, which respectively appear in \cite{CS16} and \cite{LNMZ}.

\begin{alphatheorem}[Saroglou, \cite{CS16}]
    If 
    $|A^{1-t}B^{t}| \geq |A|^{1-t}|B|^t$ holds for all origin symmetric convex bodies $A, B \subset \R^n$ and all $t\in [0,1]$, then
    \[
        \mu(A^{1-t}B^t) \geq \mu^{1-t}(A)\mu^{t}(B)
    \]
    whenever $A,B$ are origin symmetric convex bodies and $\mu$ is an even log-concave measure.
\end{alphatheorem}

\begin{alphatheorem}[Livshyts, Marsiglietti, Nayar and Zvavitch, \cite{LNMZ}]
    For $\mu$ an even log-concave measure such that 
    \[
        \mu(A^{1-t} B^t) \geq \mu^{1-t}(A)\mu^t(B)
    \]
    holds for all $t \in [0,1]$ and $A,B \subset \R^n$ origin symmetric convex bodies, then for $p >0$
    \[
        \mu((1-t)\cdot A +_p t \cdot B) \geq \mathcal{M}_{\frac{p}{n}}^{(t)}(\mu(A), \mu(B))
    \]
    holds for all $t \in [0,1]$ and $A,B$ origin symmetric convex bodies.
\end{alphatheorem}

As consequence of these results and Theorem~\ref{thm:equivalenoflpBM}, we obtain the following equivalent form of Conjecture~\ref{c:lpbm}. 

\begin{cor}\label{cor:equivalence}
    The log-Brunn-Minkowski conjecture is true if and only if its functional extension holds.  Moreover the log-Brunn-Minkowski inequality implies that the $L_p$-Borell-Brascamp-Lieb inequalities hold for all even unimodal  functions with respect to even log-concave  measures.  
\end{cor}

In other words, if the $\log$-Brunn-Minkowski inequality holds, then
for every
     $\alpha \geq - \frac{p}{n}$ and $h,f,g: \mathbb{R}^n \to [0,\infty)$ with $f$ and $g$ even and unimodal, such that
    \[
        h(z) \geq \mathcal{M}_\alpha^{(t)}(f(x),g(y))
    \]
    for $$(x,y) \in C_p(z;f,g) \coloneqq \{ (u,v) : z \in (1-t)\cdot \{f \geq f(u) \} +_p t \cdot \{ g \geq g(v) \},$$ then for any even log-concave measure $\mu$, 
    \[
        \mu( h ) \geq \mathcal{M}_r^{(t)} \left( \mu(f), \mu(g) \right),
    \]
    where $r = \frac{p\alpha}{p+n \alpha}$.


    



 \begin{rem}  As we mentioned in the introduction, the log-Brunn-Minkowski conjecture in the unconditional case follows from an application of the Prèkopa-Leindler inequality, in \cite{CS15}. It is in fact more immediate to see this through the functional analogue we introduced. Indeed, let $f,g:\mathbb{R}^n\to [0,+\infty)$ be unconditional unimodal functions and $h:\mathbb{R}^n\to [0,+\infty)$ satisfying $h(z)\geq f(u)^{1-t}g(v)^t$ for each $z\in \mathbb{R}^n$ and $u,v$ such that $z\in \{f \geq f(u)\}^{1-t}  \{g \geq g(v)\}^t$. Then, if $u,v\in [0,+\infty)^n$ and $z:=(u_1^{1-t}v_1^t,\ldots,u_n^{1-t}v_n^t)$, for each $\theta\in S^{n-1}$ and $\varepsilon_i:=sgn(\theta_i)$, setting $\tilde{u}=(\varepsilon_1u_1,\ldots,\varepsilon_nu_n), \tilde{v}=(\varepsilon_1v_1,\ldots,\varepsilon_nv_n)$ it holds
   \[
   \langle z,\theta\rangle =\sum_{i=1}^n \varepsilon_i |u_i\theta_i|^{1-t}|v_i\theta_i|^t\leq \langle \tilde{u},\theta\rangle^{1-t}\langle \tilde{v},\theta\rangle^t .
   \]
Since by unconditionality $\tilde{u}\in \{f \geq f(u)\} $ and $\tilde{v}\in \{g \geq g(v)\}$, we conclude that $z \in \{f \geq f(u)\}^{1-t}  \{g \geq g(v)\}^t$
   (equivalently $(u,v) \in C_0(z,f,g)$) and thus $h(u_1^{1-t}v_1^t,\ldots,u_n^{1-t}v_n^t)\geq f(u)^{1-t}g(v)^t$. The wanted functional inequality now follows from the multiplicative Prèkopa-Leindler inequality which may be found for example in the paper of Uhrin \cite{Uhrin}; this follows by considering $\tilde{f}(t_1,\ldots,t_n):=e^{t_1+\ldots+t_n}f(e^{t_1},\ldots,e^{t_n})$, same for $g$ and $h$, and using the classical Prèkopa-Leindler. 
   
   \end{rem}

\subsection{The case $p > 1$}

In this subsection we describe the $p >1$ a Borell-Brascamp-Lieb inequality that can be derived for $L_p$-Minkowski sums.  The ideas are similar but easier.  These results should be compared to \cite{RX21,RX23}.  Here $q$ is the H\"older conjugate to $p$ satisfying $q = \frac{p}{p-1}.$



\begin{defn} Let $p >1$, $\alpha \in [-\infty,\infty]$ and Borel measurable functions $f_0,f_1 \colon \R^n \to [0,\infty)$ and fixed $t\in [0,1]$.
For $f=(f_0,f_1)$, define 
\[
\square_{p,\alpha} f(z)= \sup_{(x,y) \in C_p(z)} \mathcal{M}_{\alpha}^{(t)}(f_0(x),f_1(y)),
\]
where
\[
C_p(z) := \left\{(x,y) \colon \text{ there exists } \lambda \in [0,1] \text{ with } z = (1-t)^\frac{1}{p}(1-\lambda)^{\frac{1}{q}}x + t^\frac{1}{p}\lambda^\frac{1}{q}y \right\}.
\]
\end{defn}
Note that in the $p > 1$ case, for Borel $A$ and $B$ then 
\begin{align} \label{eq: Lp sum for p >1}
(1-t) \cdot A +_p t \cdot B = \Phi([0,1] \times A \times B)
\end{align}
under the map $\Phi(\lambda,x,y) = (1-t)^\frac{1}{p}(1-\lambda)^{\frac{1}{q}}x + t^\frac{1}{p}\lambda^\frac{1}{q}y$.
We have the following result. 

\begin{prop}\label{p:agreementofLpsups} Let $p >1$, $\alpha \geq - \frac p n$, $t \in (0,1)$, and $f=(f_0, f_1)$, with $f_0,f_1 \colon \R^n \to [0,\infty)$ Borel measurable functions, then if
\[
    h(z) \geq \mathcal{M}_\alpha^{(t)}(f_0(x),f_1(y))
\]
for $(x,y) \in C_p(z),$
\[
    \int_{\mathbb{R}^n} h \geq \mathcal{M}_{\alpha'}^{(t)} \left( \int_{\mathbb{R}^n} f_0, \int_{\mathbb{R}^n} f_1 \right),
\]
where $\alpha' = \frac{\alpha p}{p + \alpha n}$.
\end{prop}

\begin{proof}The proof follows analogously to the results in Section \ref{sec:mesures}.  Start with the case $\alpha = - \frac{p}{n}$, where $\square_{p, \alpha} f$ is checked to be a sup-convolution  through the argument of Theorem~\ref{thm: everything is sup convolution}. Indeed, monotonicity is obvious, measurability follows from \eqref{eq: Lp sum for p >1} which realizes $(1-t) \cdot A +_p t \cdot B$ as a the continuous image of a Polish space (namely $[0,1] \times A \times B$ when $A$ and $B$ are assumed to be Borel), and the reverse triangle inequality follows from the same argument as before; for fixed $z$, find $(x,y)$ that $\mathcal{M}_{- \frac{p}{n}}^{(t)}(g_0(x),g_1(y))$ approximates $\square g(z)$ and by the hypothesis $\mathcal{M}_{- \frac{p}{n}}^{(t)}(f_0(x),f_1(y)) = \square f(z)$, and let the reverse triangle inequality for $\mathcal{M}_{- \frac  p n}^{(t)}$ conclude.  Using Theorem \ref{thm: abstract sup convolution theorem} we obtain the result for $\alpha = - \frac{p}{n}$.  For $\alpha > - \frac{p}{n}$, homogeneity again allows us to conclude.  
\end{proof}

\section{An Abstract Approach to the Equivalence of Geometric and Functional Inequalities} \label{sec:Abstract}

In this section the goal is to create an abstract framework for the equivalence between a generic inequality of Brunn-Minkowski type and its functional counterpart, with respect to more general means.  We show that a relevant sup-convolution for the Borell-Ehrhard inequality does unfortunately not yield an equivalence of the functional and geometric version of the inequality.  We close with a Borell-Brascamp-Lieb type inequality for nilpotent Lie groups.

In what follows we will consider measure spaces $(E_1, \mu_1),\ldots,(E_n, \mu_n)$ and $(F,\nu)$, and set as before $E = E_1 \times \cdots \times E_n$ and $\mu = \mu_1 \otimes \cdots \otimes \mu_n$.

\begin{defn}\label{def:generalized means}
Let $m_i \geq \mu_i(E_i)$, for $i=1,\ldots,n$, $I:=\prod_{i=1}^n[0,m_i]$ and $J:=\prod_{i=1}^n[0,m_i^2]$. We consider three generalized means
\[
    \mathcal{W},\mathcal{M}:I \to [0,\infty), \text{ and }\mathcal{N}: J \to [0,\infty),
\]
that vanish whenever a coordinate is $0$ and satisfy the following properties:
\begin{enumerate}
\item \label{item: reverse Minkowski of M}
for $u, v \in I$ such that $u+v\in I$,
\[
    \mathcal{M}(u + v) \geq \mathcal{M}(u) + \mathcal{M}(v);
\]

\item \label{item: radial convexity of N}
for $u \in J$ and $\lambda \in (0,1)$,
\[
    \mathcal{N}(\lambda u) + \mathcal{N}((1-\lambda)u) \geq \mathcal{N}(u);
\]

\item \label{item: Holder inequality for means} for $u, v \in I$, a ``H\"older" inequality is satisfied in the sense
\[
    \mathcal{W}(u) \mathcal{M}(v) \geq \mathcal{N}(uv),
\]
where $(uv)_i = u_iv_i$. 
\end{enumerate}
\end{defn}
\begin{rem}
Usually one just considers $E_i=F$ for each $i$, while $m_i\in \{1,+\infty\}$. An important example, when $m_i=+\infty$ is supplied by 
\[
    \mathcal{W}(u) \coloneqq \left( \sum_{i=1}^n t_i u_i^\alpha \right)^{\frac 1 \alpha}
\]
for $t_i > 0$ and $\alpha \geq -1$,
\[
\mathcal{M}(u) \coloneqq \left( \sum_{i=1}^n t_i u_i^p \right)^{\frac 1 p}
\]
for $p \in [-\alpha, 1]$, and 
\[
    \mathcal{N}(u) = \left( \sum_{i=1}^n t_i u_i^r \right)^{\frac 1 r}.
\]
\end{rem}

\begin{defn} To each point $z \in F$, we assume that there exists a subset of $ C(z) \subset E$. For any $f= (f_1,\dots,f_n)$, with $f_i:E_i \to (0,m_i]$, we define the $C$-sup-convolution, $\square f: F \to \mathbb{R}$ by 
    \[
        \square f(z) \coloneqq \sup_{x \in C(z)} \mathcal{M}(f(x)).
    \]
\end{defn}



\begin{prop} \label{prop: square is monotone}
  Let $\mathcal{M}$ be as above and $C(z):=\Phi^{-1}(z)$ for a continuous function $\Phi:E\to F$. Then the $C$-sup-convolution is a sup-convolution in the sense of Definition \ref{def: sup convolution}.
\end{prop}

\begin{proof}
    For $x \in C(z)$, since by assumption $f_i(x_i) \leq g_i(x_i)$ we have by monotonicity of means,
    \[
        \mathcal{M}(f_i(x_i)) \leq \mathcal{M}(g_i(x_i)).
    \]
    Taking supremums concludes the monotonicity of $\square$. 

    To prove the superadditivity and measurability properties, we proceed exactly as in Theorem \ref{thm: everything is sup convolution}. For the former we employ now property (\ref{item: reverse Minkowski of M}) from Definition \ref{def:generalized means}.
\end{proof}


\begin{defn}
    For $A = (A_1, \dots, A_n)$ with $A_i \subseteq E_i$ define
    \[
        m(A) \coloneqq m(A_1, \dots, A_n) \coloneq \{z \in F: \exists x \in (A_1\times \cdots \times A_n) \cap C(z) \}.
    \]
\end{defn}

\begin{defn}\label{def: abstract geometric inequality}
    For $\mathcal{A} = (\mathcal{A}_1, \dots, \mathcal{A}_n)$ where each $\mathcal{A}_i$ is a class of subsets on $E_i$ , we say that $(E,\mu)$ and $(F,\nu)$ satisfy a $\mathcal{A}$-restricted geometric inequality if $A=(A_1,\ldots,A_n) \in \mathcal{A}$ implies that
    \begin{equation} \label{eq: abstract geometric inequality}
        \nu(m(A)) \geq \mathcal{W}( \mu_i(A_i)).
    \end{equation}
\end{defn}



\begin{defn}
    The measure spaces $(E,\mu)$ and $(F,\nu)$ satisfy an $\mathcal{F}(\mathcal{A})$-functional inequality if $f_i \in \mathcal{F}(\mathcal{A}_i)$ implies that
    \[
        \nu_*(  \square f ) \geq \mathcal{N}(\mu_i(f_i)),
    \]
where $\nu_*(f) \coloneqq sup_{\{s \leq f: simple\}} \nu(s)$.
\end{defn}



The main theorem of this section is the next result.  

\begin{thm} \label{thm:generalgeometricfunctionalequivalence}
    If the measure spaces $(E,\mu)$ and $(F,\nu)$ satisfy an $\mathcal{A}$-restricted geometric inequality, then they satisfy a $\mathcal{F}(\mathcal{A})$-functional inequality as well. Conversely, assume that the means $\mathcal{W},\mathcal{M}$ and $\mathcal{N}$ satisfy the following ``duality'' relation: For all $u\in I$,
    \begin{equation} \label{eq: dual means}
        \sup_{(a_1,\ldots,a_n)\in I}\frac{\mathcal{N}(a_1u_1,\ldots,a_nu_n)}{\mathcal{M}(a_1,\ldots,a_n)}=\mathcal{W}(u).
    \end{equation}
Then the $\mathcal{F}(A)$-functional inequality implies the $\mathcal{A}$-restricted geometric inequality.
\end{thm}

\begin{proof}
  The proof follows as in Theorem \ref{thm: everything is sup convolution}.  One reduces to step functions via monotone convergence. 
    Thus we assume that the $f_i$ take only finitely many values and prove the result through induction on the total number of non-zero values taken by the $f_i$.  
    We write as in the proof of Theorem \ref{thm: abstract sup convolution theorem}
    \[
        f_i = \sum_{j_i =1}^{n_i} a_{j_i} \mathbbm{1}_{A_{j_i}},
    \]
    and a decomposition into $f_i^+$ and $f_i^-$, which allow,
    \begin{align*}
        \nu( \square f) 
            &\geq
                \nu( \square f^-) + \nu(\square f^+)
                    \\
            &\geq
                \mathcal{N}(\mu_i(f_i^-)) + \mathcal{N}(\mu_i(f_i^+))
                    \\
            &=
                \mathcal{N}( \lambda  \mu_i(f_i)) + \mathcal{N}((1-\lambda) \mu_i(f_i))
                    \\
            &\geq
                \mathcal{N}( \mu_i(f_i)),
    \end{align*}
    where we have used the definition of $\lambda$ in the equality and Item \eqref{item: radial convexity of N} to conclude.  To finish, for $f_i = a_i \mathbbm{1}_{A_i}$, $\square f = \mathcal{M}(a_i) \mathbbm{1}_{m(A)}$ so that
    \[
        \mu( \square f) = \mathcal{M}(a_i) \mu( m(A)) \geq \mathcal{M}(a_i) \mathcal{W}(\mu_i(A_i)) \geq \mathcal{N}(a_i \mu_i(A_i)) = \mathcal{N}(\mu_i(f_i)).
    \]

    For the converse, consider $A_i\in \mathcal{A}_i$ and $a_i\in (0,m_i]$. The $\mathcal{F}(\mathcal{A})$-functional inequality for the functions $f_i:=a_i\mathbbm{1}_{A_i}$ implies
    \[
    \nu(m(A))\mathcal{M}(a_1,\ldots,a_n)\geq \mathcal{N}(a_1\mu_1(A_1),\ldots,a_n\mu_n(A_n))
    \]
    and thus 
    \[
    \nu(m(A))\geq \sup_{a_i\in (0,m_i]}\frac{\mathcal{N}\left(a_1\mu_1(A_1),\ldots,a_n\mu_n(A_n)\right)}{\mathcal{M}(a_1,\ldots,a_n)}=\mathcal{W}(\mu_1(A_1),\ldots,\mu_n(A_n)).
    \]
\end{proof}

Numerous Brunn-Minkowski inequalities have been suggested in the setting of groups. In particular, for the Heisenberg group in \cite{Monti03, LM05, BKS2018}, and extensions thereof \cite{Bobkov2011,Tao2011,Pozuealo22}. Our machinery shows Tao's Pr\'ekopa-Leindler inequality \cite[Theorem~4]{Tao2011} for nilpotent groups is equivalent to its set theoretic formulation, and consequently to the following. Given a simply connected nilpotent Lie group $G$ and subsets $A,B$, we define their product set to be 
\[
A \cdot B = \{x \cdot y \colon x \in A, y \in B\},
\]
where $\cdot$ denotes the group product operation. For more details on nilpotent Lie groups see \cite{Tao2011,Pozuealo22}.

\begin{thm} \label{thm: nilpotent BBL}
    Let $G$ be a simply connected nilpotent Lie group of (topological) dimension $d$ and with a Haar measure $\mu$. Then, for $\alpha \in  [- \frac 1 d,1]$ and $t \in (0,1)$, if $f,g,h \colon G \to [0,\infty)$ are $\mu$-measurable functions such that
    \[
        h(x\cdot y) \geq \mathcal{M}_\alpha^{(t)}( f(x), g(y))
    \]
    holds for all $x,y$, then
    \[
        \int_G h \ d\mu \geq \mathcal{M}_{\alpha'}^{(t)} \left( \frac{\int_G  f \ d\mu}{(1-t)^d}, \frac{\int_G g \ d\mu}{t^d} \right),
    \]
    where $\alpha' = \frac{\alpha}{1 + \alpha d}$.
\end{thm}

\begin{proof}
By \cite[Theorem~3]{Tao2011} or \cite[Theorem~1.1]{Pozuealo22}, we see that
\begin{align} \label{eq: nilpotent BMI}
    \mu^{\frac 1 d} (A\cdot B) \geq \mu^{\frac 1 d} (A) + \mu^{\frac 1 d}(B).
\end{align}
Taking 
\[
f \square g (z) = \sup_{(x,y) \in C(z)} \mathcal{M}_{\alpha}^{(t)}(f(x),g(y))
\]
with $C(z) = \{(x,y) : x\cdot y =z \}$, we have a sup-convolution, as $(x,y) \mapsto x \cdot y$ is continuous, Theorem \ref{thm: everything is sup convolution} applies.  By Theorem \ref{thm: abstract sup convolution theorem} it suffices to prove the result when $f = a \mathbbm{1}_A$ and $g = b \mathbbm{1}_B$ which reduces, after applying \eqref{eq: nilpotent BMI}, to the following
\begin{align*}
    \int_G f \square_\alpha g d\mu
        &=
            \mathcal{M}_{\alpha}^{(t)} (a,b) \mu(A \cdot B)
\\
        &\geq
            \mathcal{M}_{\alpha}^{(t)} (a,b) \left(\mu^{\frac 1 d} (A) + \mu^{\frac 1 d}(B) \right)^d
                \\
        &=
            \mathcal{M}_{\alpha}^{(t)} (a,b) \mathcal{M}_{\frac 1 d}^{(t)} \left( \frac{\mu(A)}{(1-t)^d}, \frac{\mu(B)}{t^d} \right)
                \\
        &\geq
            \mathcal{M}_{\alpha'}^{(t)} \left( \frac{\mu(f)}{(1-t)^d}, \frac{\mu(g)}{t^d} \right)
\end{align*}
with $\frac 1 \alpha + d = \frac 1 {\alpha'}$ so that $\alpha' = \frac{\alpha}{1 + \alpha d}$. 
\end{proof}
Note for $\alpha >1$, the relevant operator $\square$ does not define a sup-convolution, but mirroring the Euclidean case, one can obtain for instance for any $\alpha \geq - \frac 1 d,$
\begin{align*}
    h(x \cdot y ) 
        &\geq 
            \mathcal{M}_\alpha^{(t)}(f(x),g(y))
            \\
        &= 
            \mathcal{M}_\alpha^{(t)} \left( \frac{f(x)}{\mu(f)} \mu(f), \frac{g(y)}{\mu(g)} \mu(g) \right)
                \\
        &\geq 
            \mathcal{M}_{ - \frac 1 d}^{(t)} \left( \frac{f(x)} {\mu(f)},\frac{g(y)}{\mu(g)} \right) \mathcal{M}_{\alpha'} (\mu(f), \mu(g)),
\end{align*}
where $\frac 1 \alpha = \frac 1 {\alpha'} - d$, so that $\alpha' = \frac{\alpha}{1 + \alpha d}$.  Applying the $\alpha = - \frac 1 d$ result gives 
\[
    \int_G h \ d\mu \geq \mathcal{M}_{-\infty}^{(t)} \left( \frac{1}{(1-t)^{d}}, \frac{1}{t^d} \right) \mathcal{M}_{\alpha'}^{(t)} (\mu(f), \mu(g)) = \left(\frac{1}{\max\{ 1-t, t \} } \right)^{d}  \mathcal{M}_{\alpha'}^{(t)} \left( \int_G f \ d\mu , \int_G g \ d\mu \right),
\]
an analogous but weaker inequality in the case that $\alpha \leq 1$.

\par \vspace{2em}
We take this moment to consider an example of a Brunn-Minkowski-type inequality that doesn't fit into our framework.

\begin{ex}  For $t \in \R$, define 
\[
\Phi(t):=\Phi_\gamma(t)=\frac{1}{\sqrt{2\pi}}\int_{-\infty}^t e^{-\frac{t^2}{2}}dt.
\]
For each $t\in (0,1)$, define the mean $\mathcal{M}_\gamma^{(t)}=\mathcal{M}_\gamma$ on $[0,1]^2$ 
via
\[
\mathcal{M}_\gamma(a):= \Phi\left(t\Phi^{-1}(a_1)+(1-t)\Phi^{-1}(a_2) \right).
\]
The Borell-Ehrhard inequality (\cite{EHR1, EHR1.5,EHR2, BOr03}) asserts that
\[
t\Phi^{-1}(\gamma(A))+(1-t)\Phi^{-1}(\gamma(B))\leq \Phi^{-1}\left(\gamma(tA+(1-t)B)\right).
\]
For $A=[\Phi^{-1}(a_1),\Phi^{-1}(a_2)]$ and $B=[\Phi^{-1}(b_1),\Phi^{-1}(b_2)]$ it gives
\[
t\Phi^{-1}(a_2-a_1)+(1-t)\Phi^{-1}(b_2-b_1)\leq \Phi^{-1}\left(\Phi(t\Phi^{-1}(a_2)+(1-t)\Phi^{-1}(b_2) )-\Phi(t\Phi^{-1}(a_1)+(1-t)\Phi^{-1}(b_1))\right),
\]
or equivalently, 
\[
\Phi\left(t\Phi^{-1}(a_2-a_1)+(1-t)\Phi^{-1}(b_2-b_1)\right)\leq \Phi(t\Phi^{-1}(a_2)+(1-t)\Phi^{-1}(b_2) )-\Phi(t\Phi^{-1}(a_1)+(1-t)\Phi^{-1}(b_1)).
\]
In particular, the mean $\mathcal{M}_\gamma$ satisfies Property (\ref{item: reverse Minkowski of M}) of Definition \ref{def:generalized means} and the associated operation is indeed a sup-convolution, when $C(z)=\{(x,y):\ z=tx+(1-t)y\}$.

On the other hand, let $f=a_1\mathbbm{1}_{(-\infty,\Phi^{-1}(\lambda)]}$ and $g=a_2\mathbbm{1}_{(-\infty,\Phi^{-1}(\lambda)]}$, then for $h=\mathcal{M}_\gamma(a)\mathbbm{1}_{(-\infty,\Phi^{-1}(\lambda)]}$ the functional version of Borell-Ehrhard inequality gives
\[
\Phi^{-1}\left(\lambda\mathcal{M}_\gamma(a)\right)\geq t\Phi^{-1}(\lambda a_1)+(1-t)\Phi^{-1}(\lambda a_2),
\]
or in other words, 
\[
\lambda\mathcal{M}_\gamma(a)\geq \mathcal{M}_\gamma (\lambda a).
\]
This implies that the mean $\mathcal{M}_\gamma$ does not satisfy property (\ref{item: radial convexity of N}) of Definition \ref{def:generalized means}. This shows that the Borell-Ehrhard inequality does not fit the present framework.

\end{ex}

\section*{Acknowledgments}
This project was inspired in part by Gautam Aishwarya's enthusiasm for Corollary~\ref{thm: Gaussian Borell Brascamp Lieb} and its proof. The second author thanks Aishwarya and  HIM in Bonn Germany for hosting these discussions during the Dual Trimester Program ``Boolean Analysis in Computer Science".  He thanks Tomasz Tkocz for encouragement and helpful discussions as well. We would like to thank Franck Barthe for pointing out to us Uhrin's reference \cite{Uhrin}, and Alexandros Eskenazis for bringing to our attention an error in an early version of this note and for helpful comments; in particular, for bringing to our attention how our results are connected to \cite[Question~19]{CE2025}. 

Malliaris and Roberto are supported by the FP2M federation (CNRS FR 2036); this work received support from the Graduate School EUR-MINT
(State support managed by the National Research Agency for Future
Investments program bearing the reference ANR-18-EURE-0023).  
Melbourne is supported by CONAHCYT Grant CBF2023-2024-3907. 
Roysdon is supported by US NSF Grant DMS-2548742 (formerly classified as US NSF Grant DMS-2452384).

\bibliographystyle{plain}
\bibliography{bibibi}

\end{document}